\documentclass{article}
\usepackage{amsmath}
\usepackage{amsfonts}
\usepackage{amssymb}
\usepackage{graphicx}

\newtheorem{theorem}{Theorem} [section]

\newtheorem{corollary}[theorem]{Corollary}

\newtheorem{definition}[theorem]{Definition}

\newtheorem{lemma}[theorem]{Lemma}

\newtheorem{proposition}[theorem]{Proposition}

\newenvironment{proof}[1][Proof]{\noindent\textbf{#1.} }{\ \rule{0.5em}{0.5em}}
\begin{document}

\author{Vadim E. Levit\\Department of Computer Science and Mathematics\\Ariel University Center of Samaria, ISRAEL\\levitv@ariel.ac.il
\and Eugen Mandrescu\\Department of Computer Science\\Holon Institute of Technology, ISRAEL\\eugen\_m@hit.ac.il}
\date{} 
\title{Interval greedoids and families of local maximum stable sets of graphs}
\maketitle

\begin{abstract}
A \textit{maximum stable set }in a graph $G$ is a stable set of maximum
cardinality. $S$ is a \textit{local maximum stable set} of $G$, and we write
$S\in\Psi(G)$, if $S$ is a maximum stable set of the subgraph induced by
$S\cup N(S)$, where $N(S)$ is the neighborhood of $S$.

Nemhauser and Trotter
Jr. \cite{NemhTro}, proved that any $S\in\Psi(G)$ is a subset of a maximum
stable set of $G$. In \cite{LevMan2} we have shown that the family $\Psi(T)$
of a forest $T$ forms a greedoid on its vertex set. The cases where $G$ is
bipartite, triangle-free, well-covered, while $\Psi(G)$ is a greedoid, were
analyzed in \cite{LevMan45}, \cite{LevMan07}, \cite{LevMan08}, respectively.

In this paper we demonstrate that if the family $\Psi(G)$ of the graph $G$
satisfies the accessibility property, then $\Psi(G)$ forms an interval
greedoid on its vertex set. We also characterize those graphs whose families of local maximum stable sets are either antimatroids or matroids.


\end{abstract}

\section{Introduction}

Throughout this paper $G=(V,E)$ is a simple (i.e., a finite, undirected,
loopless and without multiple edges) graph with vertex set $V=V(G)$ and edge
set $E=E(G)$.

If $X\subset V$, then $G[X]$ is the subgraph of $G$ spanned by
$X$. $K_{n},C_{n},P_{n}$ denote respectively, the complete graph on $n\geq1$
vertices, the chordless cycle on $n\geq3$ vertices, and the chordless path on
$n\geq2$ vertices. The \textit{neighborhood} of a vertex $v\in V$ is the set
$N(v)=\{w:w\in V$ \ \textit{and} $vw\in E\}$. For $A\subset V$, we denote
\[
N_{G}(A)=\{v\in V-A:N(v)\cap A\neq\varnothing\}
\]
and $N_{G}[A]=A\cup N(A)$, or shortly, $N(A)$ and $N[A]$, if no ambiguity.

If $\left\vert N(v)\right\vert =1$, then $v$ is a \textit{pendant vertex} of
$G$; $\mathrm{pend}(G)$ is the set of all pendant vertices of $G$, and by
$\mathrm{isol}(G)$ we mean the set of all isolated vertices of $G$. If $N[v]$
is a clique, i.e., $G[N[v]]$ a complete subgraph in $G$, then $v$ is a
\textit{simplicial vertex} of $G$, and $\mathrm{simp}(G)$ denotes the set
$\{v:v\in V(G)$ \textit{and} $v$ \textit{is simplicial in} $G\}$. A graph $G$
is called \textit{simplicial} if every vertex of $G$ is a simplicial vertex or
is adjacent to a simplicial vertex of $G$. A \textit{simplex} of $G$ is a
maximal clique containing at least a simplicial vertex. The simplicial graphs
were introduced by Cheston \textit{et al.}, in \cite{ChesHaLas}.

\begin{theorem}
\label{th7}\cite{ChesHaLas} If $G$ is a simplicial graph and $Q_{1},...,Q_{s}
$ are its simplices, then
\[
V(G)=V(Q_{1})\cup V(Q_{2})\cup...\cup V(Q_{s})\text{ and }s=\theta
(G)=\alpha(G),
\]
where $\theta(G)$ is the minimum number of cliques that cover $V(G)$.
\end{theorem}

A \textit{stable} set in $G$ is a set of pairwise non-adjacent vertices. A
stable set of maximum size will be referred to as a \textit{maximum stable
set} of $G$, and the \textit{stability number }of $G$, denoted by $\alpha(G)
$, is the cardinality of a maximum stable set in $G$. Let $\Omega(G)$ stand
for the set of all maximum stable sets of $G$.

The following characterization of a maximum stable set of a graph, due to
Berge, will be used in the sequel.

\begin{theorem}
\cite{Berge}\label{th5} A stable set $S$ belongs to $\Omega(G)$ if and only if
every stable set of $G$, disjoint from $S$, can be matched into $S$.
\end{theorem}

A set $A\subseteq V(G)$ is a \textit{local maximum stable set} of $G$ if
$A\in\Omega(G[N[A]])$, \cite{LevMan2}; by $\Psi(G)$ we denote the set of all
local maximum stable sets of the graph $G$. For instance, any stable set
$S\subseteq\mathrm{simp}(G)$ belongs to $\Psi(G)$, while the converse is not
generally true; e.g., $\{a\},\{e,d\}\in\Psi(G)$ and $\{e,d\}\cap
\mathrm{simp}(G)=\varnothing$, where $G$ is the graph in Figure \ref{fig10}.

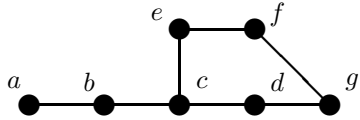
\begin{figure}[h]
\setlength{\unitlength}{1.0cm} \begin{picture}(5,1.5)\thicklines
\multiput(4.5,0)(1,0){5}{\circle*{0.29}}
\multiput(6.5,1)(1,0){2}{\circle*{0.29}}
\put(4.5,0){\line(1,0){4}}
\put(6.5,1){\line(1,0){1}}
\put(6.5,0){\line(0,1){1}}
\put(7.5,1){\line(1,-1){1}}
\put(4.3,0.3){\makebox(0,0){$a$}}
\put(5.3,0.3){\makebox(0,0){$b$}}
\put(6.8,0.3){\makebox(0,0){$c$}}
\put(7.8,0.3){\makebox(0,0){$d$}}
\put(6.2,1.2){\makebox(0,0){$e$}}
\put(7.8,1.2){\makebox(0,0){$f$}}
\put(8.8,0.3){\makebox(0,0){$g$}}
\end{picture}
\caption{A graph {with diverse local maximum stable sets}.}
\label{fig10}
\end{figure}

The following theorem concerning maximum stable sets in general graphs, due to
Nemhauser and Trotter Jr. \cite{NemhTro}, shows that for a special subgraph
$H$ of a graph $G$, some maximum stable set of $H$ can be enlarged to a
maximum stable set of $G$.

\begin{theorem}
\cite{NemhTro}\label{th1} Every local maximum stable set of a graph is a
subset of a maximum stable set.
\end{theorem}

Let us notice that the converse of Theorem \ref{th1} is not generally true.
For instance, $C_{n}$ has no proper local maximum stable set, for any $n\geq
4$. The graph $G$ in Figure \ref{fig10} shows another counterexample: any
$S\in\Omega(G)$ contains some local maximum stable set, but these local
maximum stable sets are of different cardinalities. As examples,
$\{a,d,f\}\in\Omega(G)$ and $\{a\},\{d,f\}\in\Psi(G)$, while for
$\{b,e,g\}\in\Omega(G)$ only $\{e,g\}\in\Psi(G)$.

\begin{definition}
\cite{BjZiegler}, \cite{KorLovSch} A \textit{greedoid} is a pair
$(V,\mathcal{F})$, where $\mathcal{F}\subseteq2^{V}$ is a non-empty set system
satisfying the following conditions:

\emph{Accessibility:} for every non-empty $X\in\mathcal{F}$ there is an $x\in
X$ such that $X-\{x\}\in\mathcal{F}$;

\emph{Exchange:} for $X,Y\in\mathcal{F},\left\vert X\right\vert =\left\vert
Y\right\vert +1$, there is an $x\in X-Y$ such that $Y\cup\{x\}\in\mathcal{F}$.
\end{definition}

It is worth observing that if $(V,\mathcal{F})$ has the accessibility property
and $S\in\mathcal{F}$, $\left\vert S\right\vert =k\geq2$, then there is a
chain
\[
\{x_{1}\}\subset\{x_{1},x_{2}\}\subset...\subset\{x_{1},...,x_{k-1}%
\}\subset\{x_{1},...,x_{k-1},x_{k}\}=S
\]
such that $\{x_{1},x_{2},...,x_{j}\}\in\mathcal{F}$, for all $j\in
\{1,...,k-1\}$. Such a chain we call an \textit{accessibility chain }of $S$.

In the sequel we use $\mathcal{F}$ instead of $(V,\mathcal{F})$, as the ground
set $V$ will be, usually, the vertex set of some graph.

\begin{theorem}
\label{th2}\cite{LevMan2} The family of local maximum stable sets of a forest
forms a greedoid on its vertex set.
\end{theorem}

Theorem \ref{th2} is not specific for forests. For instance, the family
$\Psi(G)$ of the graph $G$ in Figure \ref{Fig101} is a greedoid.

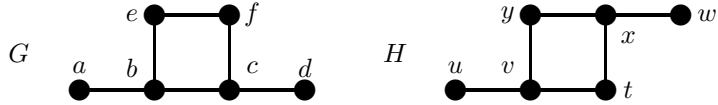
\begin{figure}[h]
\setlength{\unitlength}{1.0cm} \begin{picture}(5,1.5)\thicklines
\multiput(3,0)(1,0){4}{\circle*{0.29}}
\multiput(4,1)(1,0){2}{\circle*{0.29}}
\put(3,0){\line(1,0){3}}
\put(4,1){\line(1,0){1}}
\multiput(4,0)(1,0){2}{\line(0,1){1}}
\put(3,0.3){\makebox(0,0){$a$}}
\put(3.7,0.3){\makebox(0,0){$b$}}
\put(5.3,0.3){\makebox(0,0){$c$}}
\put(6,0.3){\makebox(0,0){$d$}}
\put(3.7,1){\makebox(0,0){$e$}}
\put(5.3,1){\makebox(0,0){$f$}}
\put(2.2,0.5){\makebox(0,0){$G$}}
\multiput(8,0)(1,0){3}{\circle*{0.29}}
\multiput(9,1)(1,0){3}{\circle*{0.29}}
\put(8,0){\line(1,0){2}}
\put(9,1){\line(1,0){2}}
\multiput(9,0)(1,0){2}{\line(0,1){1}}
\put(8,0.3){\makebox(0,0){$u$}}
\put(8.7,0.3){\makebox(0,0){$v$}}
\put(10.3,0){\makebox(0,0){$t$}}
\put(11.35,1){\makebox(0,0){$w$}}
\put(8.7,1){\makebox(0,0){$y$}}
\put(10.3,0.7){\makebox(0,0){$x$}}
\put(7.2,0.5){\makebox(0,0){$H$}}
\end{picture}
\caption{Both $G$ and $H$ are bipartite, but only $\Psi(G)$ forms {a greedoid}.}
\label{Fig101}
\end{figure}

Notice that $\Psi(H)$ is not a greedoid, where $H$ is from Figure
\ref{Fig101}, because the accessibility property is not satisfied, e.g.,
$\{y,t\}\in\Psi(H)$, but $\{y\},\{t\}$ $\notin\Psi(H)$.

A \textit{matching} in a graph $G=(V,E)$ is a set of edges $M\subseteq E$ such
that no two edges of $M$ share a common vertex. A \textit{maximum matching} is
a matching of maximum size, denoted by $\mu(G)$. A matching is \textit{perfect} if it saturates all the vertices of the graph. A matching
\[
M=\{a_{i}b_{i}:a_{i},b_{i}\in V(G),1\leq i\leq k\}
\]
of a graph $G$ is called \textit{a uniquely restricted matching} if $M$ is the unique perfect matching of $G[\{a_{i},b_{i}:1\leq i\leq k\}]$, \cite{GolHiLew}. For instance, all the
maximum matchings of the graph $G$ in Figure \ref{Fig101} are uniquely
restricted, while the graph $H$ from the same figure has both uniquely
restricted maximum matchings (e.g., $\{uv,xw\}$) and non-uniquely restricted
maximum matchings (e.g., $\{xy,tv\}$). It turns out that this is the reason
that $\Psi(H)$ is not a greedoid, while $\Psi(G)$ is a greedoid.

\begin{theorem}
\label{th22}\cite{LevMan45} For a bipartite graph $G,$ $\Psi(G)$ is a greedoid
on its vertex set if and only if all its maximum matchings are uniquely restricted.
\end{theorem}

The case of bipartite graphs owning a unique cycle, whose family of local
maximum stable sets forms a greedoid is analyzed in \cite{LevMan5}.

Let us recall that $G$ is a \textit{K\"{o}nig-Egerv\'{a}ry graph} provided
$\alpha(G)+\mu(G)=\left\vert V(G)\right\vert $, \cite{Dem}, \cite{Ster}. As a
well-known example, any bipartite graph is a K\"{o}nig-Egerv\'{a}ry graph,
\cite{Eger}, \cite{Koen}.

\begin{figure}[h]
\setlength{\unitlength}{1.0cm} \begin{picture}(5,1.5)\thicklines
\multiput(3,0)(1,0){4}{\circle*{0.29}}
\multiput(3,1)(1,0){3}{\circle*{0.29}}
\put(3,0){\line(1,0){3}}
\put(3,1){\line(1,0){1}}
\put(4,1){\line(1,-1){1}}
\put(3,0){\line(0,1){1}}
\put(5,0){\line(0,1){1}}
\put(5.3,1){\makebox(0,0){$f$}}
\put(2.7,0.3){\makebox(0,0){$a$}}
\put(2.7,1){\makebox(0,0){$b$}}
\put(4.3,0.3){\makebox(0,0){$c$}}
\put(4.3,1){\makebox(0,0){$d$}}
\put(5.3,0.3){\makebox(0,0){$e$}}
\put(6.3,0.3){\makebox(0,0){$g$}}
\put(2,0.5){\makebox(0,0){$G$}}
\multiput(8.5,0)(1,0){4}{\circle*{0.29}}
\multiput(8.5,1)(1,0){4}{\circle*{0.29}}
\put(8.5,0){\line(1,0){3}}
\put(8.5,0){\line(1,1){1}}
\put(8.5,1){\line(1,0){3}}
\put(10.5,0){\line(0,1){1}}
\put(7.5,0.5){\makebox(0,0){$H$}}
\end{picture}
\caption{$\Psi(G)$ is not a greedoid, $\Psi(H)$ is a greedoid.}
\label{fig2922}
\end{figure}
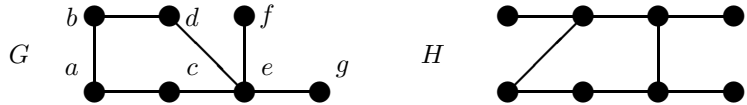

The graphs from Figure \ref{fig2922} are non-bipartite K\"{o}nig-Egerv\'{a}ry
graphs, and all their maximum matchings are uniquely restricted. Let us remark
that both graphs are also triangle-free, but only $\Psi(H)$ is a greedoid. It is clear that $\{b,c\}\in$ $\Psi(G)$, while $G[N[\{b,c\}]]$ is not a K\"{o}nig-Egerv\'{a}ry graph. As one can see from the following theorem, this observation is the real reason for $\Psi(G)$ not to be a greedoid.

\begin{theorem}
\label{th33}\cite{LevMan07} If $G$ is a triangle-free graph, then the
following assertions are equivalent:

\emph{(i)} $\Psi(G)$ is a greedoid;

\emph{(ii)} all maximum matchings of $G$ are uniquely restricted and the
closed neighborhood of every local maximum stable set of $G$ induces a
K\"{o}nig-Egerv\'{a}ry graph.
\end{theorem}

Various cases of well-covered graphs whose families of local maximum stable
sets form greedoids, were treated in \cite{LevMan07Buch}, \cite{LevMan08},
\cite{LevMan08a}, \cite{LevMan08b}.

Let $X$ be a graph with $V(X)=\{v_{i}:1\leq i\leq n\}$, and $\{H_{i}:1\leq
i\leq n\}$ be a family of graphs. Joining each $v_{i}\in V(X)$ to all the
vertices of $H_{i}$, we obtain a new graph, called the \textit{corona} of $X$
and $\{H_{i}:1\leq i\leq n\}$ and denoted by $G=X\circ\{H_{1},H_{2}%
,...,H_{n}\}$. For instance, see Figure \ref{fig12}. If $H_{1}=H_{2}=...=H_{n}=H$, we write $G=X\circ H$, and in this case, $G$ is called the \textit{corona} of $X$ and $H$.

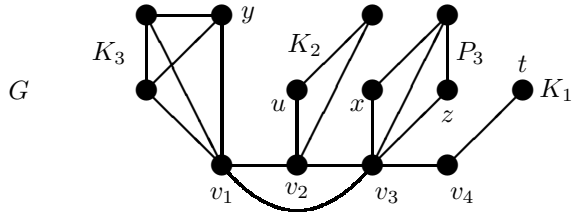
\begin{figure}[h]
\setlength{\unitlength}{1cm}\begin{picture}(5,2.8)\thicklines
\multiput(6,0.5)(1,0){4}{\circle*{0.29}}
\put(6,2.5){\circle*{0.29}}
\multiput(5,1.5)(0,1){2}{\circle*{0.29}}
\multiput(7,1.5)(1,0){4}{\circle*{0.29}}
\multiput(8,2.5)(1,0){2}{\circle*{0.29}}
\multiput(6,0.5)(1,0){3}{\line(1,0){1}}
\multiput(7,0.5)(1,0){2}{\line(0,1){1}}
\multiput(7,0.5)(1,0){2}{\line(1,2){1}}
\multiput(7,1.5)(1,0){2}{\line(1,1){1}}
\multiput(8,0.5)(1,0){2}{\line(1,1){1}}
\put(5,1.5){\line(1,1){1}}
\put(5,1.5){\line(0,1){1}}
\put(5,2.5){\line(1,0){1}}
\put(5,1.5){\line(1,-1){1}}
\put(5,2.5){\line(1,-2){1}}
\put(6,0.5){\line(0,1){2}}
\put(9,1.5){\line(0,1){1}}
\qbezier(6,0.5)(7,-0.7)(8,0.5)
\put(7.8,1.3){\makebox(0,0){$x$}}
\put(6.35,2.5){\makebox(0,0){$y$}}
\put(9,1.15){\makebox(0,0){$z$}}
\put(6.75,1.3){\makebox(0,0){$u$}}
\put(10,1.85){\makebox(0,0){$t$}}
\put(6,0.1){\makebox(0,0){$v_1$}}
\put(7,0.17){\makebox(0,0){$v_2$}}
\put(8.2,0.1){\makebox(0,0){$v_3$}}
\put(9.2,0.1){\makebox(0,0){$v_4$}}
\put(4.5,2){\makebox(0,0){$K_3$}}
\put(7.1,2.1){\makebox(0,0){$K_2$}}
\put(9.3,2){\makebox(0,0){$P_3$}}
\put(10.45,1.5){\makebox(0,0){$K_1$}}
\put(3.3,1.5){\makebox(0,0){$G$}}
\end{picture}
\caption{$G=(G[\{v_{1},v_{2},v_{3},v_{4}\}])\circ\{K_{3},K_{2},P_{3},K_{1}\}$ is a
well-covered graph.}
\label{fig12}
\end{figure}

\begin{theorem}
\label{th333}\cite{LevMan08b} If $G=X\circ\{H_{1},H_{2},...,H_{n}\}$ and
$H_{1},H_{2},...,H_{n}$ are non-empty graphs, then $\Psi(G)$ is a greedoid if
and only if every $\Psi(H_{i}),i=1,2,...,n$, is a greedoid.
\end{theorem}

If each $H_{i}$ is a complete graph, then $X\circ\{H_{1},H_{2},...,H_{n}\}$ is
called the \textit{clique corona }of $X$ and $\{H_{1},H_{2},...,H_{n}\}$;
notice that the clique corona is well-covered graph (and very well-covered,
whenever $H_{i}=K_{1},1\leq i\leq n$). Recall that $G$ is
\textit{well-covered} if all its maximal stable sets have the same cardinality, \cite{Plummer}, and $G$ is \textit{very well-covered} if, in addition, it has no isolated vertices and $\left\vert V(G)\right\vert=2\alpha(G)$, \cite{Favaron}.

\begin{corollary}
\cite{LevMan08}, \cite{LevMan08a} If $G$ is the clique corona of $X$ and
$\{H_{1},H_{2},...,H_{n}\}$, then $\Psi(G)$ is a greedoid, for any graph $X$.
\end{corollary}

In this paper we show that for any graph $G$, the family $\Psi(G)$ satisfies
the accessibility property if and only if $\Psi(G)$ is an interval greedoid.
We also prove that: $\Psi(G)$ is an antimatroid if and only if $G$ is a unique
maximum stable set whose $\Psi(G)$ satisfies the accessibility property, and
$\Psi(G)$ forms a matroid if and only if $G$ is a simplicial graph and every
non-simplicial vertex belongs to at least two different simplices.

\section{Separating examples}

Let us recall definitions of some classes of greedoids, \cite{BjZiegler}.

A \textit{matroid} is a greedoid $(V,\mathcal{F})$ enjoying the
\textit{hereditary property}:
\[
\text{\textit{if}}\ X\in\mathcal{F\ }\text{\textit{and}}\ Y\subset
X\text{,}\ \text{\textit{then}}\ Y\in\mathcal{F}.
\]

An \textit{antimatroid} is a greedoid $(V,\mathcal{F})$ \textit{closed under
union}:
\[
\text{\textit{if}}\ X,Y\in\mathcal{F}\text{\textit{,\ then}}\ X\cup
Y\in\mathcal{F}.
\]

A \textit{trimmed matroid} is the intersection of a matroid and an antimatroid.

An \textit{interval greedoid} is a greedoid $(V,\mathcal{F})$ satisfying the
following condition:
\[
\text{\textit{for\ every}}\ X\in\mathcal{F\ }\text{\textit{the\ family}%
}\ \{Y\in\mathcal{F}:Y\subseteq X\}\ \text{\textit{is\ an\ antimatroid.}}%
\]

A \textit{local poset greedoid} is a greedoid $(V,\mathcal{F})$ satisfying the property:
\[
\text{\textit{if\ }}X,Y,Z\in\mathcal{F\ }\text{\textit{and}}\ X,Y\subset
Z\text{\textit{,\ then}}\ X\cup Y,X\cap Y\in\mathcal{F}.
\]

The following result helps us to emphasize a number of separating examples.

\begin{lemma}
\label{lem4}If $\Omega(G)=\{S\}$, then $S-\{x\}\in\Psi(G)$ holds for any $x\in
S$.
\end{lemma}

\begin{proof}
Let us suppose that $S-\{x\}\notin\Psi(G)$ is true for some $x\in S$. It
follows that there exists $A\in\Omega(G[N[S-\{x\}]])$ with $\left\vert
A\right\vert >\left\vert S-\{x\}\right\vert =\alpha(G)-1$. Hence, we obtain
that $A=S$ which implies $x\in N(S-\{x\})$, in contradiction with the fact
that $x\in S$.
\end{proof}

Let us remark that Lemma \ref{lem4} is not necessarily true when two or more
vertices are deleted from the unique maximum stable set; e.g., if
$\Omega(P_{2k+1})=\{S\}$, then $\mathrm{pend}(P_{2k+1})\subseteq S$, while
$S-\mathrm{pend}(P_{2k+1})\notin\Psi(P_{2k+1})$, for any $k\geq2$.

\begin{itemize}
\item Let us observe that
\[
\mathcal{F}=\{\emptyset,\{a\},\{b\},\{c\},\{a,b\},\{a,c\},\{a,b,c\}\}
\]
is a greedoid on $\{a,b,c\}$, but there is no graph $G$ such that
$\Psi(G)=\mathcal{F}$, because, according to Lemma \ref{lem4}, $\{a,b,c\}\in
\mathcal{F}$ implies that $\{b,c\}\in\mathcal{F}$, as well.

\item Let us notice that
\[
\mathcal{F}=\{\emptyset
,\{a\},\{c\},\{a,b\},\{a,c\},\{c,d\},\{a,b,c\},\{a,c,d\},\{a,b,c,d\}\}
\]
is an antimatroid on $\{a,b,c,d\}$, but there is no graph $G$ such that
$\Psi(G)=\mathcal{F}$, because, according to Lemma \ref{lem4}, $\{a,b,c,d\}\in
\mathcal{F}$ implies that $\{a,b,d\}\in\mathcal{F}$, too. Consequently, we
infer also that there is an interval greedoid $\mathcal{F}$, such that
$\mathcal{F}\neq\Psi(G)$ is true for any graph $G$.

\item If $G=\overline{K_{n}}$, then $\Psi(G)$ produces both a matroid, an
antimatroid and a local poset greedoid. The same is true for some trees, e.g.,
for $P_{3}$.

\begin{figure}[h]
\setlength{\unitlength}{1.0cm} \begin{picture}(5,0.7)\thicklines
\multiput(4,0)(1,0){6}{\circle*{0.29}}
\put(4,0){\line(1,0){5}}
\put(4,0.4){\makebox(0,0){$a$}}
\put(5,0.4){\makebox(0,0){$b$}}
\put(6,0.4){\makebox(0,0){$c$}}
\put(7,0.4){\makebox(0,0){$d$}}
\put(8,0.4){\makebox(0,0){$e$}}
\put(9,0.4){\makebox(0,0){$f$}}
\end{picture}
\caption{A tree $T$ whose $\Psi(T)$ is neither a matroid nor an antimatroid.}
\label{fig41}
\end{figure}
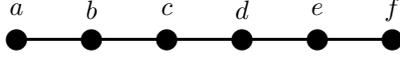

\item The family of maximum local stable sets of the tree $P_{6}$ (see Figure
\ref{fig41}) is not a matroid because while $\left\{  a,c\right\}  \in
\Psi(P_{6})$, the set $\left\{  c\right\}  $ does not belong to $\Psi(P_{6})$.
The family $\Psi(P_{6})$ is not an antimatroid, too. One of the reasons is
that while $\left\{  a,c\right\}  ,\left\{  d,f\right\}  \in\Psi(P_{6})$, the
set $\left\{  a,c\right\}  \cup\left\{  d,f\right\}  $ is not even stable.

\item It is also easy to check that: $\Psi(P_{5})$ is an antimatroid and not a
matroid; $\Psi(P_{2})$ is a matroid, but it is not an antimatroid.

\item If $G=P_{4}$ or $G=K_{1,n},n\geq1$, then $\Psi(G)$ is a local poset greedoid.

\item $\Psi(P_{5})$ is a greedoid, but it is not a local poset greedoid. To
see that, let us consider $X=\{a,b\},Y=\{b,c\},Z=\{a,b,c\}$, that clearly
satisfy
\[
X,Y,Z\in\Psi(P_{5}),X\subset Z,Y\subset Z,X\cup Y\in\Psi(P_{5}),
\]
but $X\cap Y=\{b\}\notin\Psi(P_{5})$.

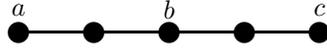
\begin{figure}[h]
\setlength{\unitlength}{1.0cm} \begin{picture}(5,0.7)\thicklines
\multiput(5,0)(1,0){5}{\circle*{0.29}}
\put(5,0){\line(1,0){4}}
\put(5,0.3){\makebox(0,0){$a$}}
\put(7,0.3){\makebox(0,0){$b$}}
\put(9,0.3){\makebox(0,0){$c$}}
\end{picture}
\caption{$\Psi(P_{5})$ is a greedoid, but not a local poset greedoid.}
\label{fig23}
\end{figure}

\item Let $V(P_{4})=\{a,b,c,d\},E(P_{4})=\{ab,bc,cd\}$. Then, $\Psi(P_{4})$ is
a greedoid, but is neither a matroid, since
\[
\{a,c\}\in\Psi(P_{4})\text{, but }\{c\}\notin\Psi(P_{4}),
\]
nor an antimatroid, because
\[
\{a,c\},\{b,d\}\in\Psi(P_{4})\text{, while }\{a,b,c,d\}\notin\Psi(P_{4}).
\]
On the other hand, the family
\[
M=\{\varnothing,\{a\},\{b\},\{c\},\{d\},\{a,c\},\{a,d\},\{b,c\},\{b,d\}\}
\]
is a matroid, the family
\[
AM=\{\{a\},\{d\},\{a,c\},\{a,d\},\{b,d\},\{a,b,d\},\{a,c,d\},\{a,b,c,d\}\}
\]
is an antimatroid, and $\Psi(P_{4})=M\cap AM$, i.e., $\Psi(P_{4})$ is a
trimmed matroid.
\end{itemize}

\section{An interval greedoid on vertex set of a graph}

Let us observe that the family $\Psi(G)$ is not generally closed under
intersection or difference, even if $G$ has a unique maximum stable set. For
instance, the tree $P_{7}$ in Figure \ref{fig42} has a unique maximum stable
set, namely $\{a,c,e,g\}$, and while
\[
A=\{a,c\},B=\{a,d\},C=\{c,e,g\}\in\Psi(P_{7}),
\]
none of the sets $A-B,A\cap C$ belong to $\Psi(P_{7})$.

However, if every connected component of G is a complete graph, then $\Psi(G)
$ is obviously closed under intersection or difference. As far as the union
operation is concerned, we have the following general statement.

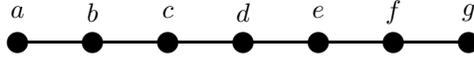
\begin{figure}[h]
\setlength{\unitlength}{1.0cm} \begin{picture}(5,0.7)\thicklines
\multiput(4,0)(1,0){7}{\circle*{0.29}}
\put(4,0){\line(1,0){6}}
\put(4,0.4){\makebox(0,0){$a$}}
\put(5,0.4){\makebox(0,0){$b$}}
\put(6,0.4){\makebox(0,0){$c$}}
\put(7,0.4){\makebox(0,0){$d$}}
\put(8,0.4){\makebox(0,0){$e$}}
\put(9,0.4){\makebox(0,0){$f$}}
\put(10,0.4){\makebox(0,0){$g$}}
\end{picture}
\caption{A tree $T$ with a unique maximum stable set: $\{a,c,e,g\}$.}
\label{fig42}
\end{figure}

\begin{theorem}
\label{th3}For any graph $G$, if $A,B\in\Psi(G)$ and $A\cup B$ is stable, then
$A\cup B\in\Psi(G)$.
\end{theorem}

\begin{proof}
For $S\in\Omega(N[G[A\cup B]])$ let us denote:
\begin{align*}
S_{A}  &  =S\cap(N[A]-N[A\cap B]),\\
\quad S_{B}  &  =S\cap(N[B]-N[A\cap B]),\\
S_{AB}  &  =S\cap N[A\cap B].
\end{align*}
Since $A,B\in\Psi(G)$, it follows also that
\[
\left\vert S_{A}\right\vert +\left\vert S_{AB}\right\vert \leq\left\vert
A\right\vert \text{ and }\left\vert S_{B}\right\vert +\left\vert
S_{AB}\right\vert \leq\left\vert B\right\vert .
\]
On the other hand, $\left\vert S_{AB}\right\vert \geq\left\vert A\cap
B\right\vert $, because otherwise, $S_{A}\cup(A\cap B)\cup S_{B}$ is stable in
$N[A\cup B]$ with $\left\vert S_{A}\cup(A\cap B)\cup S_{B}\right\vert
>\left\vert S\right\vert $, in contradiction with the choice $S\in
\Omega(N[G[A\cup B]])$. Consequently, we obtain:
\[
\left\vert S_{A}\right\vert +\left\vert S_{AB}\right\vert +\left\vert
S_{B}\right\vert +\left\vert A\cap B\right\vert \leq\left\vert S_{A}%
\right\vert +2\left\vert S_{AB}\right\vert +\left\vert S_{B}\right\vert
\leq\left\vert A\right\vert +\left\vert B\right\vert
\]
which implies:
\[
\left\vert S\right\vert =\left\vert S_{A}\right\vert +\left\vert
S_{AB}\right\vert +\left\vert S_{B}\right\vert \leq\left\vert A\right\vert
+\left\vert B\right\vert -\left\vert A\cap B\right\vert =\left\vert A\cup
B\right\vert .
\]
Hence, we get that $A\cup B\in\Omega(G[N[A\cup B]])$, i.e., $A\cup B\in
\Psi(G)$.
\end{proof}

\begin{figure}[h]
\setlength{\unitlength}{1.0cm} \begin{picture}(5,1.5)\thicklines
\multiput(5.5,0)(1,0){3}{\circle*{0.29}}
\multiput(5.5,1)(1,0){4}{\circle*{0.29}}
\put(5.5,0){\line(0,1){1}}
\put(5.5,0){\line(1,1){1}}
\put(6.5,0){\line(0,1){1}}
\put(6.5,0){\line(1,1){1}}
\put(6.5,1){\line(1,-1){1}}
\put(7.5,0){\line(0,1){1}}
\put(7.5,0){\line(1,1){1}}
\put(5.5,1.3){\makebox(0,0){$a$}}
\put(6.5,1.3){\makebox(0,0){$b$}}
\put(7.5,1.3){\makebox(0,0){$c$}}
\put(8.5,1.3){\makebox(0,0){$d$}}
\end{picture}\caption{A graph satisfying $A\cap\mathrm{simp}(G)\neq\emptyset$
for every $A\in\Psi(G)$.}
\label{fig1010}
\end{figure}
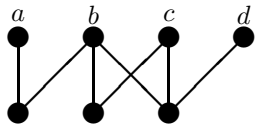

The condition \textquotedblright$A\cap\mathrm{simp}(G)\neq\emptyset$, for any
$A\in\Psi(G)$\textquotedblright\ is clearly necessary, but is not sufficient
to guarantee the accessibility property for the family $\Psi(G)$; e.g., the
graph $G$ in Figure \ref{fig1010} has $\{a,b,c\}\in\Psi(G),\{a,b,c\}\cap
\mathrm{simp}(G)=\{a\}$, but no subset consisting of two elements of
$\{a,b,c\}$ belongs to $\Psi(G)$.

It is worth observing that if $\Psi(G)$ has the accessibility property and
$S\in\Psi(G)$, $\left\vert S\right\vert =k\geq2$, then there is a chain
\[
\{x_{1}\}\subset\{x_{1},x_{2}\}\subset...\subset\{x_{1},...,x_{k-1}%
\}\subset\{x_{1},...,x_{k-1},x_{k}\}=S
\]
such that $\{x_{1},x_{2},...,x_{j}\}\in\Psi(G)$, for all $j\in\{1,...,k-1\}$.
Such a chain we call an \textit{accessibility chain }of $S$.

\begin{theorem}
\label{th4}If the family $\Psi(G)$ of a graph $G$ satisfies the accessibility
property, then the following assertions are true:

\emph{(i)} $\Psi(G)$ forms a greedoid on its vertex set;

\emph{(ii)} $\Psi(G)$ is an interval greedoid.
\end{theorem}

\begin{proof}
\emph{(i)} We have to prove that $\Psi(G)$ satisfies also the exchange property.

Let $A,B\in$ $\Psi(G)$ such that $\left\vert B\right\vert =\left\vert
A\right\vert +1=m+1$. Hence, there is an accessibility chain for $B$, say
\[
\{b_{1}\}\subset\{b_{1},b_{2}\}\subset...\subset\{b_{1},...,b_{m}\}\subset B.
\]

Since $B$ is stable, $A\in$ $\Psi(G)$ but $\left\vert A\right\vert <\left\vert
B\right\vert $, it follows that there exists some $b\in B-A$, such that
$b\notin N[A]$.

If $b=b_{1}$, then
\[
A\cup\{b_{1}\} \leq\alpha(N[A\cup\{b_{1}\}])=\alpha(N[A]\cup N[\{b_{1}%
\}])\leq\alpha(N[A]) + \alpha(N[\{b_{1}\}])=|A|+1=|A\cup\{b_{1}\}|,
\]
because $b_{1}$ is a simplicial vertex and $A\cup\{b_{1}\}$ is a stable set.
Consequently, $A\cup\{b_{1}\}\in\Psi(G)$.

Otherwise, let $b_{k+1}\in B, k\geq1$ be the first vertex in $B$ satisfying
the conditions:
\[
b_{1},...,b_{k}\in N[A]\text{ and }b_{k+1}\notin N[A].
\]
Since $\{b_{1},...,b_{k}\}$ is stable in $G[N[A]]$ and $A\in\Omega(G[N[A]])$,
Theorem \ref{th5} implies that there is a matching $M$ from $\{b_{1}%
,...,b_{k}\}-A$ into $A$, i.e., there is $\{a_{1},...,a_{k}\}\subseteq A$ such
that for any $i\in\{1,...,k\}$ either $a_{i}=b_{i}$ or $a_{i}b_{i}\in M$.

We show that $A\cup\{b_{k+1}\}\in\Psi(G)$.

If not, there exists some $\{c_{1},...,c_{p},d_{1},...,d_{s}\}$ in
$\Omega(G[N[A\cup\{b_{k+1}\}]])$ such that:
\[
p+s\geq m+2,\{c_{1},...,c_{p}\}\subseteq N[A]\text{ and }\{d_{1}%
,...,d_{s}\}\subseteq N(b_{k+1}).
\]
Since $\{b_{1},...,b_{k+1}\}$ is in $\Psi(G)$, $\{a_{1},...,a_{k}%
,d_{1},...,d_{s}\}\subseteq N[\{b_{1},...,b_{k+1}\}]$, while $\{a_{1}%
,...,a_{k}\}$ and $\{d_{1},...,d_{s}\}$ are stable sets, it follows that
\[
\left\vert \{d_{1},...,d_{s}\}\cap N[\{a_{1},...,a_{k}\}]\right\vert \geq s-1,
\]
because otherwise $\{a_{1},...,a_{k},d_{1},...,d_{s}\}$ contains some stable
set of $k+2$ vertices, contradicting the fact that
\[
\{b_{1},...,b_{k+1}\}\in\Omega(G[N[\{b_{1},...,b_{k+1}\}]]).
\]
So, we may suppose that $\{d_{1},...,d_{s-1}\}\subseteq N[\{a_{1}%
,...,a_{k}\}]$. Since
\[
\{c_{1},...,c_{p}\}\subset N[A]\text{ and }\{d_{1},...,d_{s-1}\}\subseteq
N[\{a_{1},...,a_{k}\}],
\]
it follows that
\[
W=\{c_{1},...,c_{p},d_{1},...,d_{s-1}\}\subseteq N[A]
\]
and $W$ is a stable set of size
\[
\left\vert W\right\vert =p+s-1\geq m+1,
\]
i.e., $W$ is larger than $A$, in contradiction with the choice $A\in\Psi(G)$.

\emph{(ii)} For $A\in\Psi(G)$ let us denote
\[
\Psi(A)=\{B\in\Psi(G):B\subseteq A\}.
\]
Since, by part \emph{(i)}, $\Psi(G)$ is a greedoid, it is clear that $\Psi(A)
$ is also a greedoid. For any $B_{1},B_{2}$ belonging to $\Psi(A)$, the set
$B_{1}\cup B_{2}$ is stable, because $A$ is stable. According to Theorem
\ref{th3}, it follows that $B_{1}\cup B_{2}\in\Psi(A)$. Hence, $\Psi(A)$ is an
antimatroid and consequently, $\Psi(G)$ is an interval greedoid.
\end{proof}

As a consequence, we may say that all the greedoids we have obtained by
Theorems \ref{th2}, \ref{th22}, \ref{th33}, and \ref{th333}, are interval greedoids.

\begin{corollary}
The family $\Psi(G)$ of a graph $G$ satisfies the accessibility property if
and only if $\Psi(G)$ forms an interval greedoid.
\end{corollary}

\section{The graphs whose $\Psi(G)$ is either an antimatroid or a matroid}

If $\left\vert \Omega(G)\right\vert =1$, then $G$ is called a \textit{unique
maximum stable set graph}, \cite{Gunther}, \cite{HopSta}, \cite{levm1},
\cite{SieToppVolk}.

\begin{lemma}
\label{lem1}$G$ is a unique maximum stable set graph if and only if $\Psi(G) $
is closed under union.
\end{lemma}

\begin{proof}
Let $\Omega(G)=\{S\}$ and $A,B\in\Psi(G)$. By Theorem \ref{th1}, both $A$ and
$B$ are subsets of $S$. Hence, $A\cup B$ is a stable set in $G$, and according
to Theorem \ref{th3}, we infer that $A\cup B\in\Psi(G)$.\mathstrut

Conversely, let $\Psi(G)$ be closed under union. If $\Omega(G)$ contains two
different elements, say $S_{1},S_{2}$, then $S_{1},S_{2}\in\Psi(G)$ and
consequently, $S_{1}\cup S_{2}\in\Psi(G)$. Hence, $S_{1}\cup S_{2}$ must be a
stable set in $G$, in contradiction with $\left\vert S_{1}\cup S_{2}%
\right\vert >\alpha(G)$.
\end{proof}

Notice that the graphs $G_{1},G_{2}$ from Figure \ref{Fig22} are unique
maximum stable set graphs, but only $\Psi(G_{1})$ does not satisfy the
accessibility property, since $\{y,z\}\in\Psi(G_{1})$, while $\{y\},\{z\}$ do
not belong to $\Psi(G_{1})$. Hence, by Theorem \ref{th4}, only $\Psi(G_{2})$
is a greedoid. Moreover, the following theorem shows that $\Psi(G_{2})$ is
even an antimatroid.

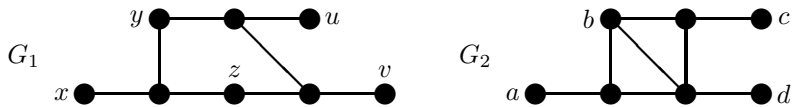
\begin{figure}[h]
\setlength{\unitlength}{1.0cm} \begin{picture}(5,1.2)\thicklines
\multiput(3,0)(1,0){5}{\circle*{0.29}}
\multiput(4,1)(1,0){3}{\circle*{0.29}}
\put(3,0){\line(1,0){4}}
\put(4,1){\line(1,0){2}}
\put(4,0){\line(0,1){1}}
\put(5,1){\line(1,-1){1}}
\put(2.7,0){\makebox(0,0){$x$}}
\put(3.7,1){\makebox(0,0){$y$}}
\put(5,0.3){\makebox(0,0){$z$}}
\put(6.3,1){\makebox(0,0){$u$}}
\put(7,0.3){\makebox(0,0){$v$}}
\put(2.2,0.5){\makebox(0,0){$G_{1}$}}
\multiput(9,0)(1,0){4}{\circle*{0.29}}
\multiput(10,1)(1,0){3}{\circle*{0.29}}
\put(9,0){\line(1,0){3}}
\put(10,1){\line(1,0){2}}
\put(10,0){\line(0,1){1}}
\put(10,1){\line(1,-1){1}}
\put(11,0){\line(0,1){1}}
\put(8.7,0){\makebox(0,0){$a$}}
\put(9.7,1){\makebox(0,0){$b$}}
\put(12.3,1){\makebox(0,0){$c$}}
\put(12.3,0){\makebox(0,0){$d$}}
\put(8.2,0.5){\makebox(0,0){$G_{2}$}}
\end{picture}
\caption{$\Omega(G_{i})=\{S_{i}\},i=1,2$, where ${S}_{1}{=\{x,y,z,u,v\}}
${\ and }$S_{2}=\{a,b,c,d\}$.}
\label{Fig22}
\end{figure}

\begin{theorem}
\label{th8}For any graph $G$, the following assertions are equivalent:

\emph{(i)} $\Psi(G)$ is an antimatroid;

\emph{(ii)} $G$ is a unique maximum stable set graph and $\Psi(G)$ satisfies
the accessibility property.
\end{theorem}

\begin{proof}
If $\Psi(G)$ is an antimatroid, then $\Psi(G)$ satisfies the accessibility
property and is closed under union. By Theorem \ref{th3}, $G$ must be a unique
maximum stable set graph.

Conversely, since $\Psi(G)$ satisfies the accessibility property, Theorem
\ref{th4} ensures that $\Psi(G)$\ is a greedoid. Further, according to Lemma
\ref{lem1}, $\Psi(G)$\ is also closed under union, because $G$\ is a unique
maximum stable set graph. Consequently, $\Psi(G)$\ is an antimatroid.
\end{proof}

For instance, all the graphs from Figure \ref{Fig30} are unique maximum stable
graphs, but only $\Psi(G_{1})$ and $\Psi(G_{2})$ are antimatroids; $\Psi
(G_{3})$ is not a greedoid, since $\{x,y\}\in\Psi(G_{3})$, while
$\{x\},\{y\}\notin\Psi(G_{3})$.

\begin{figure}[h]
\setlength{\unitlength}{1.0cm} \begin{picture}(5,1.2)\thicklines
\multiput(1,0)(1,0){4}{\circle*{0.29}}
\put(2,1){\circle*{0.29}}
\put(1,0){\line(1,0){3}}
\put(2,1){\line(1,-1){1}}
\put(2,0){\line(0,1){1}}
\put(0.2,0.5){\makebox(0,0){$G_{1}$}}
\multiput(6,0)(1,0){4}{\circle*{0.29}}
\multiput(7,1)(1,0){2}{\circle*{0.29}}
\put(6,0){\line(1,0){3}}
\put(7,0){\line(0,1){1}}
\put(8,0){\line(0,1){1}}
\put(5.2,0.5){\makebox(0,0){$G_{2}$}}
\multiput(11,0)(1,0){4}{\circle*{0.29}}
\multiput(12,1)(1,0){3}{\circle*{0.29}}
\put(11,0){\line(1,0){2}}
\put(12,1){\line(1,0){2}}
\put(12,0){\line(0,1){1}}
\put(13,0){\line(0,1){1}}
\put(13,1){\line(1,-1){1}}
\put(11.7,1){\makebox(0,0){$x$}}
\put(13.3,0){\makebox(0,0){$y$}}
\put(10.2,0.5){\makebox(0,0){$G_{3}$}}
\end{picture}
\caption{$G_{1},G_{2}$ and $G_{3}$ are unique maximum stable graphs.}
\label{Fig30}
\end{figure}
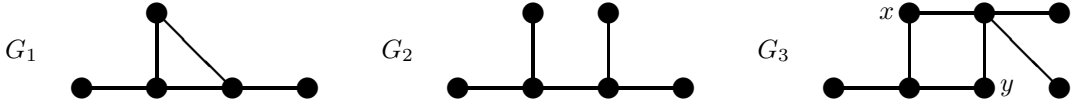

\begin{corollary}
\label{cor2}If $T$ is a tree, then the following assertions are equivalent:

\emph{(i)} $\Psi(T)$ is an antimatroid;

\emph{(ii)} $T$ is a unique maximum stable set graph;

\emph{(iii)} $T$ has a maximum stable set $S$ such that $\left\vert N(v)\cap
S\right\vert \geq2$ holds for every $v\in V(T)-S$.
\end{corollary}

\begin{proof}
The equivalence \emph{(i) }$\iff$ \emph{(ii) }follows from Theorems \ref{th8},
\ref{th2}.

The equivalence \emph{(ii) }$\iff$ \emph{(iii) }was proved in \cite{Gunther},
\cite{Zito}.
\end{proof}

As far as the graphs in Figure \ref{Fig232} are concerned, it is easy to check that:

\begin{itemize}
\item $\Psi(G_{1})$ is not a greedoid, because $\{u,v\}\in\Psi(G_{1})$, but
$\{a\},\{b\}\notin\Psi(G_{1})$;

\item $\Psi(G_{2})$ is a greedoid, but not a matroid, since $\{a,b\}\in
\Psi(G_{2})$, while $\{a\}\notin\Psi(G_{2})$;

\item $\Psi(G_{3})$ is a matroid.
\end{itemize}

\begin{figure}[h]
\setlength{\unitlength}{1.0cm} \begin{picture}(5,1.2)\thicklines
\multiput(1,0)(1,0){4}{\circle*{0.29}}
\multiput(2,1)(1,0){2}{\circle*{0.29}}
\put(1,0){\line(1,0){3}}
\put(2,1){\line(1,0){1}}
\put(1,0){\line(1,1){1}}
\put(2,0){\line(0,1){1}}
\put(3,0){\line(0,1){1}}
\put(3,1){\line(1,-1){1}}
\put(2.2,0.3){\makebox(0,0){$u$}}
\put(3.3,1.1){\makebox(0,0){$v$}}
\put(0.3,0.5){\makebox(0,0){$G_{1}$}}
\multiput(6,0)(1,0){4}{\circle*{0.29}}
\multiput(7,1)(1,0){2}{\circle*{0.29}}
\put(6,0){\line(1,0){3}}
\put(7,1){\line(1,0){1}}
\put(7,1){\line(1,-1){1}}
\put(7,0){\line(0,1){1}}
\put(8,0){\line(0,1){1}}
\put(6.8,0.3){\makebox(0,0){$a$}}
\put(8.3,1){\makebox(0,0){$b$}}
\put(5.3,0.5){\makebox(0,0){$G_{2}$}}
\multiput(11,0)(1,0){4}{\circle*{0.29}}
\multiput(11,1)(1,0){3}{\circle*{0.29}}
\put(11,0){\line(1,0){3}}
\put(11,0){\line(0,1){1}}
\put(11,1){\line(1,0){1}}
\put(11,1){\line(1,-1){1}}
\put(12,0){\line(0,1){1}}
\put(13,0){\line(0,1){1}}
\put(10.3,0.5){\makebox(0,0){$G_{3}$}}
\end{picture}
\caption{$G_{1},G_{2}$ and $G_{3}$ are simplicial graphs.}
\label{Fig232}
\end{figure}
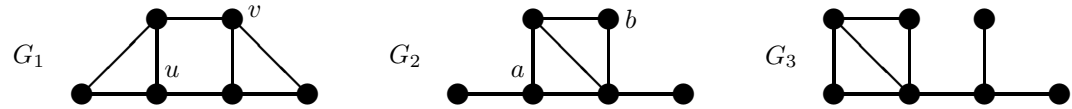

\begin{theorem}
\label{th6}If $G$ is a graph, then the following assertions are equivalent:

\emph{(i)} $\Psi(G)$ is a matroid;

\emph{(ii)} $S\subseteq\mathrm{simp}(G)$, for every $S\in$ $\Omega(G)$;

\emph{(iii)} $G$ is a simplicial graph and every non-simplicial vertex belongs
to at least two different simplices.
\end{theorem}

\begin{proof}
\emph{(i)} $\Rightarrow$\emph{(ii)} Suppose that $\Psi(G)$ is a matroid. Any
$S\in\Omega(G)$ belongs also to $\Psi(G)$, and therefore, by hereditary
property, it follows that $\{x\}\in\Psi(G)$, for every $x\in S$. Hence,
$\alpha(G[N[x]])=\left\vert \{x\}\right\vert =1$, and this ensures that $N[x]$
is a clique. Consequently, we infer that $x\in\mathrm{simp}(G)$, for each
$x\in S$. Therefore, $S\subseteq\mathrm{simp}(G)$, for every $S\in$ $\Omega(G)
$.

\emph{(ii)} $\Rightarrow$ \emph{(i) }According to Theorem \ref{th4}, it is
sufficient to show that $\Psi(G)$ has hereditary property.

Let now $S_{1}\in\Psi(G)$ and $S_{2}\subset S_{1}$. By Theorem \ref{th1},
there is some $S\in\Omega(G)$ such that $S_{1}\subset S$. Hence,
$S_{2}\subseteq\mathrm{simp}(G)$, which clearly implies that $S_{2}\in\Psi(G)$.

\emph{(ii)} $\Rightarrow$ \emph{(iii)} Suppose that $G$ is not simplicial.
Then there is at least one vertex $v\in V(G)$ such that $N[v]\cap
\mathrm{simp}(G)=\varnothing$. For each $S\in\Omega(G)$ we have $S\subseteq
\mathrm{simp}(G)$, and this implies that $S\cap N[v]$ $=\emptyset$. Hence,
$S\cup\{v\}$ is stable in $G$, in contradiction with the choice $S\in
\Omega(G)$. Therefore, $G$ is a simplicial graph.

Assume that there exists a vertex $v\in V(G)-\mathrm{simp}(G)$ such that $v$
belongs to a unique simplex, say $Q$, and let $S\in\Omega(G)$. Since
$S\subseteq\mathrm{simp}(G)$ and $v\notin\mathrm{simp}(G)$, it follows that
$S\cap Q=\{w\}\neq\{v\}$. Hence, we get that $(S\cup\{v\})-\{w\}\in\Omega(G)$,
and consequently, $(S\cup\{v\})-\{w\}\subseteq\mathrm{simp}(G)$, contradicting
the assumption that $v\notin\mathrm{simp}(G)$.

So, we may conclude that $G$ is a simplicial graph and every non-simplicial
vertex belongs to at least two different simplices.

\emph{(iii)} $\Rightarrow$ \emph{(ii)} According to Theorem \ref{th7},
\[
V(G)=V(Q_{1})\cup V(Q_{2})\cup...\cup V(Q_{s}),
\]
where $Q_{1},...,Q_{s}$ are the simplices of $G$ and $s=\theta(G)=\alpha(G)$.
Suppose that there is some $S\in\Omega(G)$ such that $S\nsubseteq
\mathrm{simp}(G)$. Let $v\in S-\mathrm{simp}(G)$ and $Q_{i},Q_{j}$ be two
different simplices of $G$, both containing $v$. Since $v\in S$ and
$Q_{i},Q_{j}$ are cliques in $G$, it follows that $S\cap Q_{i}=\{v\}=S\cap
Q_{j}$. Let $v_{i}\in Q_{i}\cap\mathrm{simp}(G)$ and $v_{j}\in Q_{j}%
\cap\mathrm{simp}(G)$ be non-adjacent vertices in $G$. Then, the set
$(S\cup\{v_{i},v_{j}\})-\{v\}$ is stable in $G$ and larger than $S$, in
contradiction with $S\in\Omega(G)$. Therefore, $S\subseteq\mathrm{simp}(G)$
must hold for each $S\in\Omega(G)$, and this completes the proof.
\end{proof}

\begin{corollary}
\label{cor1}If $G$ is a triangle-free graph, then the following statements are equivalent:

\emph{(i)} $\Psi(G)$ is a matroid;

\emph{(ii)} $S\subseteq\mathrm{pend}(G)\cup\mathrm{isol}(G)$, for every $S\in$
$\Omega(G)$;

\emph{(iii)} $G$ has as connected components: $K_{1},K_{2}$, and graphs having unique maximum stable sets, namely, sets of their pendant vertices.
\end{corollary}

\begin{proof}
Now,\textbf{\ }$\mathrm{simp}(G)=\mathrm{pend}(G)\cup\mathrm{isol}(G)$, since
$G$ is a triangle-free graph. Further, the proof follows from Theorem
\ref{th6}.
\end{proof}

Since bipartite graphs are triangle-free, Corollary \ref{cor1} is true for
bipartite graphs, as well. It is easy to see that $\Psi(K_{1})$ and
$\Psi(K_{2})$ are matroids. For trees with more than three vertices, we have the following result.

\begin{corollary}
If $T$ is a tree of order at least three, then the following assertions are equivalent:

\emph{(i)} $\Psi(T)$ is a matroid;

\emph{(ii)} $\mathrm{pend}(T)$ is the unique maximum stable set of $T$;

\emph{(iii)} $\Psi(T)$ is a trimmed matroid.
\end{corollary}

\begin{proof}
Corollary \ref{cor1} assures that "\emph{(i) }$\iff$ \emph{(ii)}" is valid.
Further, using Corollary \ref{cor2}, it follows that "\emph{(ii)} $\implies$
\emph{(iii)}" is also true. Clearly, \emph{(iii)} implies \emph{(i)}.
\end{proof}

If $T$ is a tree having a unique maximum stable set, then $\Psi(T)$ is a
greedoid, but is not necessarily a local poset greedoid; e.g., the tree in
Figure \ref{fig23}.

\begin{proposition}
\label{prop2}If every $S\in\Omega(G)$ is contained in $\mathrm{simp}(G)$, then
$\Psi(G)$ is a local poset greedoid.
\end{proposition}

\begin{proof}
First, $\Psi(G)$ is a greedoid, by Theorem \ref{th6}. Further, let us notice
that if a stable set $S$ is contained in $\mathrm{simp}(G)$, then $S$ belongs
to $\Psi(G)$. Therefore, for any $X,Y,Z\in\Psi(G)$ satisfying $X\subset
Z,Y\subset Z$, it follows that $X\cup Y,X\cap Y\in\Psi(G)$. Hence, $\Psi(G)$
is a local poset greedoid.
\end{proof}

Let us notice that the converse of Proposition \ref{prop2} is not true. For
instance, $\Psi(P_{4})$ is a local poset greedoid, and, clearly, there exists $S\in
\Omega(P_{4})$, which is not contained in $\mathrm{simp}(P_{4})$.

\section{Conclusions}

In this paper we have proved that in the case of the family $\Psi(G)$,
the accessibility property implies the exchange property, and the resulting
greedoids form a proper subfamily of the class of interval greedoids. The graphs,
whose families of local maximum stable sets are either antimatroids or matroids, have
been described completely.

\textbf{Open problem}: characterize the interval greedoids, the matroids, and
the antimatroids produced by $\Psi(G)$.


\begin{thebibliography}{99}

\bibitem {Berge}C. Berge, \emph{Some common properties for regularizable
graphs, edge-critical graphs and B-graphs}, in: Graph Theory and Algorithms,
Lecture Notes in Computer Science \textbf{108} (1980) 108-123,
Springer-Verlag, Berlin.

\bibitem {BjZiegler}A. Bj\"{o}rner, G. M. Ziegler, \emph{Introduction to
greedoids}, in N. White (ed.), \\ \emph{Matroid Applications}, 284-357, Cambridge University Press, 1992.

\bibitem {ChesHaLas}G. H. Cheston, E. O. Hare, S. T. Hedetniemi, R. C. Laskar,
\emph{Simplicial graphs}, Congressus Numerantium 67 (1988) 105-113.

\bibitem {Dem}R. W. Deming, \emph{Independence numbers of graphs - an
extension of the K\"{o}nig--Egerv\'{a}ry theorem}, Discrete Mathematics
\textbf{27} (1979) 23--33.

\bibitem {Eger}E. Egervary, \emph{On combinatorial properties of matrices},
Mat. Lapok \textbf{38} (1931) 16-28.

\bibitem {Favaron}O. Favaron, \emph{Very well-covered graphs}, Discrete
Mathematics \textbf{42} (1982) 177-187.

\bibitem {GolHiLew}M. C. Golumbic, T. Hirst, M. Lewenstein, \emph{Uniquely
restricted matchings}, Algorithmica \textbf{31} (2001) 139-154.

\bibitem {Gunther}G. Gunther, B. Hartnell, D.F. Rall, \emph{Graphs whose
vertex independence number is unaffected by single edge addition or deletion},
Discrete Applied Mathematics \textbf{46} (1993) 167--172.

\bibitem {HopSta}G. Hopkins, W. Staton, \emph{Graphs with unique maximum
independent sets}, \\ Discrete Mathematics \textbf{57} (1985) 245-251.

\bibitem {KorLovSch}B. Korte, L. Lov\'{a}sz, R. Schrader, \emph{Greedoids},
Springer-Verlag, Berlin, 1991.

\bibitem {Koen}D. K\"{o}nig, \emph{Graphen und Matrizen}, Mat. Lapok
\textbf{38} (1931) 116-119.

\bibitem {levm1}V. E. Levit, E. Mandrescu, \emph{On the structure of $\alpha
$-stable graphs}, Discrete Mathematics \textbf{236} (2001) 227-243.

\bibitem {LevMan5}V. E. Levit, E. Mandrescu, \emph{Unicycle bipartite graphs
with only uniquely restricted maximum matchings}, Proceedings of the Third
International Conference on Combinatorics, Computability and Logic, (DMTCS'1),
Springer, (C. S. Calude, M. J. Dinneen and S. Sburlan eds.) (2001) 151-158.

\bibitem {LevMan2}V. E. Levit, E. Mandrescu, \emph{A new greedoid: the family
of local maximum stable sets of a forest}, Discrete Applied Mathematics
\textbf{124} (2002) 91-101.

\bibitem {LevMan45}V. E. Levit, E. Mandrescu, \emph{Local maximum stable sets
in bipartite graphs with uniquely restricted maximum matchings}, Discrete
Applied Mathematics \textbf{132} (2004) 163-174.

\bibitem {LevMan07}V. E. Levit, E. Mandrescu, \emph{Triangle-free graphs with
uniquely restricted maximum matchings and their corresponding greedoids}, Discrete Applied Mathematics \textbf{155} (2007) 2414--2425.

\bibitem {LevMan07Buch}V. E. Levit, E. Mandrescu, \emph{On local maximum
stable sets of the corona of a path with complete graphs}, Proceedings of the $6^{th}$
Congress of Romanian Mathematicians, University of Bucharest, Romania (2007) (in press).

\bibitem {LevMan08}V. E. Levit, E. Mandrescu, \emph{Well-covered graphs and
greedoids}, Proceedings of the $14^{th}$ Computing: The Australasian Theory
Symposium (CATS2008), Wollongong, NSW, Conferences in Research and Practice in
Information Technology Volume \textbf{77} (2008) 89-94.

\bibitem {LevMan08a}V. E. Levit, E. Mandrescu, \emph{The clique corona
operation and greedoids}, Combinatorial Optimization and Applications, Second
International Conference, COCOA 2008, Lecture Notes in Computer Science
\textbf{5165} (2008) 384-392.

\bibitem {LevMan08b}V. E. Levit, E. Mandrescu, \emph{Graph operations that are
good for greedoids}, MOPTA 2008, University of Guelph, Guelph, Canada, prE-print arXiv:0809.1806v1 (2008) 9 pp.

\bibitem {NemhTro}G. L. Nemhauser, L. E. Trotter, Jr., \emph{Vertex packings:
structural properties and algorithms}, Mathematical Programming \textbf{8}
(1975) 232-248.

\bibitem {Plummer}M. D. Plummer, \emph{Some covering concepts in graphs},
Journal of Combinatorial Theory \textbf{8} (1970) 91-98.

\bibitem {SieToppVolk}W. Siemes, J. Topp, L. Volkman, \emph{On unique
independent sets in graphs}, \\ Discrete Mathematics \textbf{131} (1994) 279-285.

\bibitem {Ster}F. Sterboul, \emph{A characterization of the graphs in which
the transversal number equals the matching number}, Journal of Combinatorial
Theory Series B \textbf{27} (1979) 228--229.

\bibitem {Zito}J. Zito, \emph{The structure and maximum number of maximum
independent sets in trees}, Journal of Graph Theory \textbf{15} (1991) 207-221.
\end{thebibliography}
\end{document}